\batchmode
\documentclass{amsart}
\usepackage{xspace}
\usepackage{amssymb}
\usepackage{epsfig}
\usepackage{latexsym}
\usepackage{amsmath}
\usepackage{amsfonts}
\usepackage{amsthm}

\numberwithin{equation}{section}

\newtheorem{thm}{Theorem}[section]
\newtheorem{lem}[thm]{Lemma}
\newtheorem{prop}[thm]{Proposition}
\newtheorem{cor}[thm]{Corollary}
\newtheorem{conj}[thm]{Conjecture}

\newcommand{\T}{\operatorname{T}}
\newcommand{\h}{\operatorname{ht}}

\newcommand{\R}{\mathbb R}
\newcommand{\C}{\mathbb C}
\newcommand{\N}{\mathbb N}
\newcommand{\Q}{\mathbb Q}
\newcommand{\SL}{\operatorname{SL}}

\newcommand{\Z}{ \mathbb{Z}}

\begin{document}
\title{Entropy and escape of mass for Hilbert modular spaces}

\author{Shirali Kadyrov}

\thanks{The author acknowledges support by the SNF (200021-127145).}

\begin{abstract}
We study the relation between metric entropy and escape of mass for the Hilbert modular spaces with the action of a diagonal element.
\end{abstract}

\maketitle

\section{Introduction}
Many interesting problems are related to equidistribution on homogeneous spaces. Often the ambient space is not compact, which leads to the question whether the limit measure is still a probability measure. One case of this so-called non-escape of mass problem is for a sequence of measures that are invariant under one parameter unipotent subgroups. In this case the answer is simple: for a sequence of invariant and ergodic measures under unipotent subgroups the limit measure is still a probability measure or the zero measure \cite{MozSha}. This fact relies on the quantitative non-divergences estimates for
 unipotents due to works of S.~G.~Dani~\cite{Dani} (further refined by G.~A.~Margulis and D.~Kleinbock~\cite{MarKle}). 
 
 In this paper we are interested in the dynamics of diagonal flows. Consider a sequence of probability measures invariant under a particular diagonal element of a linear group acting on the homogeneous space. In this case, the limit measure of the space could be any value in [0,1]. However, if additionally we assume that the measures have high entropy w.r.t. the diagonal element then one can show that the limit measure is not 0. This has been realized in \cite{TO} where M.~Einsiedler, E.~Lindenstrauss, Ph.~Michel, and A.~Venkatesh show the following.
\begin{thm}
\label{thm:torusorbit} Let $X$ be the unit tangent bundle to the modular surface and $T$ be the time 1-map for the geodesic flow. Then, any sequence of $T$-invariant probability measures $\mu_n$ with entropies $h_{\mu_n}(T)\geq c$ satisfies that any weak$^*$ limit $\mu_{\infty}$ has at least $\mu_{\infty}(X) \geq 2c-1$ mass left.
\end{thm}
Here, $\mu_{\infty}$ is a weak$^*$ limit of the sequence $(\mu_n)_{n \ge 1}$ if for some subsequence $n_k$ and for all
$f \in C_c(X)$ we have $$\lim_{k\to \infty}\int_X f d\mu_{n_k} \to \int_X f d \mu_{\infty}.$$
In \cite{EinKad} M.~Einsiedler and the author prove a similar theorem for the space of three-dimensional lattices.
Our main goal in this paper is to extend Theorem~\ref{thm:torusorbit} to the following more general setup.

Let $F$ be an algebraic number field and let $\mathcal{O}$ be its ring of integers. Let $S^\infty=\{\sigma_1,...,\sigma_{r+s}\}$ be its archimedean places where $\{\sigma_1,...,\sigma_r\}$ are the real places and the rest are complex ones. Define
$$G:=\prod_{n=1}^r \SL_2(\R)\times \prod_{m=1}^s \SL_2(\C) \text{ and }\Gamma:=\SL_2(\mathcal{O}).$$
We have the natural embedding of $\Gamma$ into $G$ via 
$$\Delta:\gamma \to (\sigma_1(\gamma),\sigma_2(\gamma),\dots,\sigma_{r+s}(\gamma))$$ where $\sigma_j (\gamma)=\left( \begin{array}{cc} \sigma_j(a) &\sigma_j(b)  \\
 \sigma_j(c)&\sigma_j(d) \end{array} \right)$ for $\gamma=\left( \begin{array}{cc} a&b  \\
 c&d \end{array} \right)\in \Gamma$. Then $\Gamma$ becomes a lattice in $G$ (cf. Lemma~\ref{lem:mahler}). It is an irreducible lattice and the quotient space $X:=\Gamma \backslash G$ is non-compact. The bi-quotient $\Gamma \backslash \SL_2(\R)\times \SL_2(\R) \slash {\rm SO}(2) \times {\rm SO}(2)$ in the case of a real quadratic field $F$ over $\Q$ is known as Hilbert modular surface.
 
Let $a$ be any fixed diagonal element of $G$. Then there exist $a_j \in \R$ and $\theta_j \in [0,2\pi]$ such that 
$$a=diag(e^{i\theta_1}e^{a_1/2},e^{-i\theta_1}e^{-a_1/2})\times \cdots \times diag(e^{i\theta_{r+s}}e^{a_{r+s}/2},e^{-i\theta_{r+s}}e^{-a_{r+s}/2})$$
with $\theta_1,\dots,\theta_r=0.$ Now, we define the action of $\T$ on $X$ by $\T(x)=x\cdot a$.

In \S~\ref{sec:prelim} we define the height function $\h(\cdot)$ on $X$. Now, if we define $X_{<M}=\{x\in X \,:\, \h(x)<M\}$ then $X_{<M}$ becomes pre-compact (cf. Lemma~\ref{lem:mahler}). We similarly define $X_{\ge M}$. Now, we can state the main result.

Let $|a_1|+\cdots+|a_r|=h_r$ and $|a_{r+1}|+\cdots+|a_{r+s}|=h_s$. We note that the maximal metric entropy of $\T$ is $h_r+2h_s$, which we denoted by $h_{\max}(\T).$ 
\begin{thm}
\label{thm:main}
Let $M > \max\{e^{3h_{\max}(\T)},100\}$ be given.Then, there exists a continuous decreasing function $\phi:\R^+ \to \R$ with $\lim_{M\to \infty}\phi(M)=0$ such that 
$$\mu(X_{< M}) \ge 1-\frac{2}{h_{\max}(\T)} (h_{\max}(T)-h_{\mu}(T))+\phi(M)$$
for any $\T$-invariant probability measure $\mu$ on $X$.
In particular, for a sequence of $\T$-invariant measures $\mu_n$ with $h_{\mu_n}(\T) \ge h$ one has that any weak$^*$ limit $\mu_{\infty}$ has at least $\frac{2h}{h_{\max}(\T)}-1$ mass left.
\end{thm}
Whenever $h \in (h_{\max}(\T)/2,h_{\max}(\T)]$ there will be some mass left in the limit. We think that the theorem is sharp in the following sense: there should exists a sequence of $\T$-invariant probability measures $(\mu_n)_{n\ge 1}$ on $X$ with $\lim_{n\to \infty} h_{\mu_n}(\T)=h_{\max}(\T)/2$ such that the limit measure is the 0 measure. A similar construction has been carried out in \cite{const} for the space $\SL_n(\Z) \backslash \SL_n(\R)$ of unimodular lattices. Theorem~\ref{thm:main} suggests the following.
\begin{conj} Let $G'$ be a $\Q$-group and $\Gamma'$ be an arithmetic lattice of $\Q$-rank one. Let $T'$ be a right multiplication on $\Gamma' \backslash G'$ by a diagonalizable element in $G$. Then, any sequence of $T'$-invariant probability measures $\mu_n$ on $\Gamma' \backslash G'$ with entropies $h_{\mu_n} \ge c$ satisfies that any weak* limit $\mu_\infty$ has at least 
$$\mu_\infty(\Gamma' \backslash G') \ge 1-\frac{2}{h_{\max}(T')} (h_{\max}(T')-c)$$
 mass left where $h_{\max}(T')$ is the maximal metric entropy of $T'$.
\end{conj}
For the $\Q$-rank one case, the conjecture suggests that once the entropies of the measures are uniformly greater than $1/2$ of the maximal entropy, there is always some mass left in the limit. For a heuristic explanation we refer to Remark~5.2 in \cite{TO}.

Another interesting case studying the limits of a sequence of probability measures arises by averaging an arbitrary measure under iterates of some element of the ambient group. In this case, the notion of entropy does not make sense rather one has to consider the dimension of the measures.

For any group $H$ we define $B_{\epsilon}^{H}(g)$ to be the open ball in $H$ of radius $\epsilon>0$ centered at $g\in H$ and we simply write $B_{\epsilon}^{H}$ if the ball is centered at the identity 1. Let us consider the following subgroups of $G$
$$U^+=\{g\in G:a^{-n}g\ a^n\to 1 \text{ as } n\to -\infty\},$$
$$U^-=\{g\in G: a^{-n}g a^n\to 1 \text{ as } n\to \infty\},$$
$$L=\{g\in G: g a=a g \}.$$
We let $D:=\dim U^+ \le r+2s$. Let $d \in [0,D]$ be given and let us consider a probability measure $\nu$ in $X$ with the following property. For any $\delta>0$ there exists $\epsilon'>0$ such that for any $\epsilon<\epsilon'$ one has
$$\nu(xB_{\epsilon}^{U^+}B_\eta^{U^-L})\ll \epsilon^{d-\delta} \text{ for any } \eta \in (0,1) \text{ and for any } x\in X.$$
In this case say that $\nu$ has a \emph{dimension at least $d$ in the unstable direction}. 
Now, we consider the following sequence of measures $\mu_n$ defined by
$$\mu_n=\frac{1}{n}\sum_{j=0}^{n-1}\T^j_*\nu$$
where $\T^j_*\nu$ is the push-forward of $\nu$ under $\T^j$. We have
\begin{thm} 
\label{thm:dim}
For a fixed $d$ let $\nu$ be a probability measure of dimension at least $d$ in the unstable direction with respect to $a$, and let $\mu_n$ be as above. Then the sequence of probability measures $(\mu_n)_{n\ge 1}$ satisfies that any
weak$^*$ limit $\mu_{\infty}$ has at least $\mu_{\infty}(X) \geq  1- \frac{2a_*(D-d)}{h_{\max}(\T)}$ mass left where $a_*=\max\{|a_i|: i=1,\dots,r+s\}.$
\end{thm}
In particular, if $\nu$ has full dimension, that is if $d=D$, then the limit $\mu_{\infty}$ is a probability measure. In this case with a minor additional assumption on $\nu$ one in fact obtains the equidistribution result, that is, the limit measure $\mu_{\infty}$ is the Haar measure \cite{Shi}. We say that an element $x \in X$ is {\it divergent on average with respect to} $a$ if $\lim_{N\to\infty}\frac{1}{N}\{n\in [0,N-1] : \T^n(x) \in K\}=0$ for any compact set $K$ in $X$. 

We note that if we have a measure $\nu$ as above for some $d$ which is supported in the set of points in $X$ that diverge on average then clearly any limit $\mu_\infty$of $(\mu_n)_{\ge 1}$ is the zero measure which implies that $d \le D-\frac{h_{\max}(\T)}{2a_*}.$ This hints the following.
\begin{cor}
The Hausdorff dimension of the points in $X$ that are divergent on average w.r.t $a$ is at most $ \dim G -\frac{h_{\max}(\T)}{2a_*}.$
\end{cor}
 The proof of the corollary is easily obtained from Theorem~\ref{thm:dim} using \cite[Corollary~4.12]{Fal} and is left to the reader (cf. \cite[Corollary~1.7]{EinKad}).

In the next section we will consider some basic facts. In \S~\ref{sec:mainthm} we state the main ingredients and show how one deduces Theroem~\ref{thm:main}. In \S~\ref{sec:partitions} we introduce the partitions and count the number of elements in these partitions. In \S~\ref{sec:mainprop} we obtain the main proposition and finally, in \S~\ref{sec:dim} we indicate how one proves Theorem~\ref{thm:dim}.

{\bf Acknowledgements:} This work is part of the author's doctoral dissertation at The Ohio State University. The author would like to thank his adviser M. Einsiedler for useful conversations. He also would like to thank the referee for useful comments which in particular helped to improve the results of the previous version of the paper.
\section{Preliminaries}\label{sec:prelim}
We consider the space $X$ as a subspace of the space of $\mathcal{O}$-submodules $\Lambda$ of $(\R^2)^r\times (\C^2)^s$ with the following properties:
\begin{enumerate}
\item $\Lambda$ is an $\mathcal O$-submodule generated by two vectors $v,w$ of $(\R^2)^r\times (\C^2)^s$,
\item $v=(v_1',v_1'')\times (v_2',v_2'')\times \cdots \times (v_{r+s}',v_{r+s}'')$ and $w=(w_1',w_1'')\times (w_2',w_2'')\times \cdots \times (w_{r+s}',w_{r+s}'')$ are such that $\det\left( \begin{array}{cc} v_j'&v_j''  \\
 w_j'&w_j'' \end{array} \right)=1$ for $j=1,...,r+s$.
\end{enumerate}
From now on, we use a standard notation  $v=(v_1',v_1'')\times (v_2',v_2'')\times \cdots \times (v_{r+s}',v_{r+s}'')$ for a vector $v\in (\R^2)^r\times (\C^2)^s$. A similar notation is used for $w  \in (\R^2)^r\times (\C^2)^s.$
The action of $\mathcal{O}$ on $(\R^2)^r\times (\C^2)^s$ is given by $\lambda \cdot v=$
$$(\sigma_1(\lambda)v_1',\sigma_1(\lambda)v_1'')\times (\sigma_2(\lambda)v_2',\sigma_2(\lambda)v_2'')\times \cdots \times (\sigma_{r+s}(\lambda)v_{r+s}',\sigma_{r+s}(\lambda)v_{r+s}'')$$
for any $\lambda\in \mathcal{O}$ and any $v \in (\R^2)^r\times (\C^2)^s.$

Now, we define the \emph{height} function $\h(\cdot)$ from $X$ to $\R^+$ as follows. On $\R$ and on $\C$ we consider the usual absolute value $|\cdot |$ and for any $(v_j',v_j'')$ in $\R^2$ or in $ \C^2$ by the norm $|\cdot|$ we mean $|(v_j',v_j'')|=\max\{|v_j'|,|v_j''|\}$. For a vector $v=(v_1',v_1'')\times (v_2',v_2'')\times \cdots \times (
v_{r+s}',v_{r+s}'')$ in an $\mathcal O$-submodule $\Lambda \in X$ we define the `norm' by
 $$\|v\|=\prod_{j=1}^{r+s} |(v_j',v_j'')|^{\delta_j}$$
where
$$\delta_j = \begin{cases} 1, & \mbox{if } j\in\{1,2,\dots,r\} \\ 2, & \mbox{if } j\in\{r+1,r+2,\dots,r+s\} \end{cases}.$$ 
Now, we define the height of $\Lambda$:
 $$\h(\Lambda):=\max\{\|v\|^{-1}\,:\, v \in \Lambda-\{0\}\}.$$ 
We note that this is well defined as $\|v\| \neq 0$ whenever $v \neq 0.$\\
{\bf Definition.}
A nonzero vector $v$ in an $\mathcal O$-submodule $\Lambda$ is said to be \emph{primitive} if $(F  v) \cap \Lambda=\mathcal{O}v$.
\begin{lem}
\label{lem:shortvector}
Up to multiplication by units, for any element $\Lambda \in X$ there can be at most one primitive (short) vector of norm $< 1$.
\end{lem}
Having only one short vector is crucial throughout  the paper. Obtaining similar results as in this paper for spaces that allow more than one primitive short vectors requires different techniques (cf. \cite{EinKad}). 
\begin{proof}
Assume by contradiction that there are two distinct primitive vectors $e,f \in \Lambda$ such that $\|e\|<1,\|f\|<1$ up to multiplication by units. Let $e=(e_1',e_1'')\times (e_2',e_2'')\times \cdots \times (e_{r+s}',e_{r+s}'')$ and $f=(f_1',f_1'')\times (f_2',f_2'')\times \cdots \times (f_{r+s}',f_{r+s}'')$. We pick $v=(v_1',v_1'')\times (v_2',v_2'')\times \cdots \times (v_{r+s}',v_{r+s}'')$ and $w=(w_1',w_1'')\times (w_2',w_2'')\times \cdots \times (w_{r+s}',w_{r+s}'')$ which generate $\Lambda$ over $\mathcal O$ as a submodule and satisfy the property (ii). There are $\lambda_1,\lambda_2,\nu_1,\nu_2 \in \mathcal O$ such that $e=\lambda_1 v + \lambda_2 w$ and $f=\nu_1 v+ \nu_2 w$. We have
\[\prod_{j=1}^{r+s}\det\left(\begin{array}{cc} e_j'&e_j''\\
f_j' & f_j''\end{array}\right)^{\delta_j}=\prod_{j=1}^{r+s}\det\left(\left(\begin{array}{cc} \sigma_j(\lambda_1) &\sigma_j(\lambda_2) \\
 \sigma_j(\nu_1) &\sigma_j(\nu_2) \end{array}\right)\left(\begin{array}{cc}v_j' &v_j''\\
 w_j' &w_j'' \end{array}\right)\right)^{\delta_j}\]
where $\{\sigma_1,...,\sigma_{r+s}\}=S^\infty$. Since, $\det\left(\begin{array}{cc}v_j' &v_j''\\
 w_j' &w_j'' \end{array}\right)=1$ for $j=1,2,\dots, r+s$ we must have
$$\prod_{j=1}^{r+s}(e_j'f_j''-e_j''f_j')^{\delta_j}=\prod_{j=1}^{r+s}\sigma_j(\lambda_2 \nu_1-\lambda_1 \nu_2)^{\delta_j}.$$
We now claim that $\lambda_2\nu_1 \neq \lambda_1 \nu_2$. Otherwise, we see that $\frac{\nu_1}{\lambda_1} e=\nu_1 v+\frac{\nu_1\lambda_2}{\lambda_1}w=f$ where without loss of generality we assumed that $\lambda_1 \neq 0.$ Then, we have $f \in (F e)\cap \Lambda=\mathcal O e $ and $e \in (F f)\cap \Lambda=\mathcal O f$ which imply, upto multiplication by units, that $e$ and $f$ are the same which is a contradiction. 

Since $\lambda_1,\lambda_2,\nu_1,\nu_2 \in \mathcal{O}$, from the above claim we obtain that $$\prod_{j=1}^{r+s}\sigma_j(\lambda_2\nu_1-\lambda_1 \nu_2)^{\delta_j}=N(\lambda_2\nu_1-\lambda_1 \nu_2)\ge 1$$ where $N(\cdot)$ is the number theoretic norm. It follows that
\begin{equation}
\label{eqn:prod}
\prod_{j=1}^{r+s}(e_j'f_j''-e_j''f_j')^{\delta_j}\ge 1.
\end{equation}
From Cauchy-Schwartz inequality we see that
$$|(e_j'f_j''-e_j''f_j')| \le |(e_j',e_j'')|\cdot|(f_j',f_j'')|.$$
Hence,
$$|\prod_{j=1}^{r+s}(e_j'f_j''-e_j''f_j')^{\delta_j}| \le \prod_{j=1}^{r+s}|(e_j',e_j'')^{\delta_j}|\cdot\prod_{j=1}^{r+s}|(f_j',f_j'')^{\delta_j}|=\|e\|\|f\|<1.$$
Thus, we obtain a contradiction to (\ref{eqn:prod}). Therefore, up to multiplication by units, there can be at most one primitive short vector of norm $<1$.
\end{proof}
We will need the following well known fact (see for example \cite{Sel}).
\begin{lem}
\label{lem:mahler}
$\Gamma$ is a lattice in $G$ and $X_{<M}$ is pre-compact.
\end{lem}
The idea of the proof is to embed $G$ as a $\Q$-group in $\SL_{2(r+2s)}(\R)$. This identification goes deeper namely that the points of the module are identified with the points of the lattice and the $\|\cdot \|$ function we considered above is just the Euclidean norm on $\R^{2(r+2s)}$. This gives that $\Gamma $ is a lattice in $G$ and moreover using Mahler's compactness criterion we obtain that $X_{<M}$ is pre-compact.
\section{Main Ingredients and the Proof of Theorem~\ref{thm:main}}
\label{sec:mainthm}
In this section we will state Lemma~\ref{lem:Q_{M,N}} and Proposition~\ref{prop:mainn} without proofs and show how they can be used to deduce Theorem~\ref{thm:main}. To make use of both Lemma~\ref{lem:Q_{M,N}} and Proposition~\ref{prop:mainn} we need the following
lemma which gives an upper bound for entropy in
terms of covers by Bowen balls. 

Define a (forward) Bowen $N$-ball (of radius $\eta$) to be the translate $xB_N$ for some $x\in X$ of $$B_N=\bigcap_{n=0}^{N-1}a^{n}B_{\eta}^{G} a^{-n}$$
where $\eta>0$ is fixed such that the $\log$ map from $B_{\eta}^G$ to the Lie algebra of $G$ is injective.
\begin{lem}\label{lem:entropy} Let $\mu$ be a $\T$-invariant ergodic probability measure on $X$. For any $N\geq
1$ and $\epsilon >0$ let $BC(N,\epsilon)$ be the minimal number of (forward)
Bowen $N$-balls needed to cover any particular subset of $X$ of measure
bigger than $\epsilon$. Then
$$h_{\mu}(\T)\leq \liminf_{N \to \infty}\frac{\log
BC(N,\epsilon)}{N}.$$
\end{lem}
To prove the lemma one roughly uses the trivial entropy bound, namely
$$H(\xi)\le \log |\xi|$$
where $|\xi|$ is the number of elements of the partition $\xi$ and the existence of fine partitions with thin boundary. The proof is left to the reader which is very similar to \cite[Lemma~B.2]{TO}. 

For $M,N\ge 1$ given we define a partition
$$Q_{M,N}:=\bigvee _{n=0}^{N-1}\T^{-n}\{X_{<M},X_{\ge M}\}.$$
\begin{lem}
\label{lem:Q_{M,N}}
The partition $Q_{M,N}$ has $\ll e^{O(\frac{\log \log M}{\log M})N}$ elements for any $M\ge e^{h_{\max}(\T)}$ and $N\in \N$ where implied constants do not depend on $M,N$.
\end{lem}
Here $X \ll Z  $ means that there exists a positive constant $c$ such that $X \le c Z.$ Also, $X \ll_d Z  $ means that the constant $c$ depends on $d$. The proof of Lemma~\ref{lem:Q_{M,N}} is given in \S~\ref{sec:proofoflemma}.

The partition elements of $Q_{M,N}$ can be described by the suitable subsets of $[0,N-1]$ in the sense that for any $Q\in Q_{M,N}$ there exists $\mathcal V \subset [0,N-1]$ with integer end points such that 
$$Q=\{ x\in X :  \forall n \in [0,N-1], \T^n(x) \in X_{\ge M} \text{ if and only if } n \in \mathcal V\}.$$ 
In this case, we denote $Q$ by $Q(\mathcal V).$
\begin{prop}
\label{prop:mainn} 
For any $M> e^{3h_{\max}(\T)}$ the partition element $Q(\mathcal V) \in Q_{M,N}$ with $Q(\mathcal V) \subset X_{<M}$ can be covered by 
$$\ll_M e^{O(\frac{\log \log M}{\log M})N} e^{h_{\max}(\T)(N-\frac{|\mathcal V|}{2})}$$
 Bowen $N$-balls for any $N \in \N$ where the implied constant in $O(\cdot)$ is independent of $M,N$.
\end{prop}
The proof of Proposition~\ref{prop:mainn} easily follows from Proposition~\ref{prop:main} together with Lemma~\ref{lem:P_{M,N}} and it is given after the statement of Proposition~\ref{prop:main}.
\begin{proof}[Proof of the Theorem~\ref{thm:main}] Note first that it suffices to
consider ergodic measures. For if $\mu$ is not ergodic, we can
write $\mu$ as an integral of its ergodic components $\mu=\int
\mu_t d\tau(t)$ for some probability space $(E,\tau)$, see for example \cite[Theorem 6.2]{EinWar}. Therefore,
we have $\mu(X_{\ge M})=\int \mu_t(X_{\ge M})d\tau(t)$, but also
$h_{\mu}(\T)=\int h_{\mu_t}(\T) d\tau(t)$, see for example \cite[Thm.~8.4]{WB}, so that
desired estimate follows from the ergodic case.

Suppose that $\mu$ is ergodic.  Let $M>e^{3h_{\max}(\T)}$ be such that $\mu(X_{<M})>0$. Later in the proof we will show how one may choose $M$ independent of $\mu$ which is crucial in obtaining the last part of the theorem. We would like to apply
Lemma~\ref{lem:entropy}. For this we need to find an upper bound for
covering a subset of $X$ of measure $\epsilon$ by Bowen $N$-balls. Let us fix $\epsilon>0$ such that $\mu(X_{<M})>2\epsilon.$  The pointwise ergodic theorem implies
\[\frac{1}{N}\sum_{n=0}^{N-1}1_{X\geq M}(\T^n(x))\to \mu(X_{\geq M})\]
as $N \to \infty$ for a.e. $x \in X$. Thus, there is $N_0$ such that for $N>N_0$ the average on the
left will be bigger than $\mu(X_{\geq M})-\epsilon$ for any $x \in
X_1$ for some $X_1 \subset X$ with measure $\mu(X_1)> 1-\epsilon$.
Clearly, for any $N>N_0$ we have $\mu(Z)>\epsilon$ where
\[Z=X_1\cap X_{<M}.\]
Now, we would like to find an upper bound for the number of Bowen
$N$-balls needed to cover the set $Z$. Here $N \rightarrow \infty$
while $\epsilon$ is fixed. We now split $Z$ into the sets
$P(\mathcal V)$ as in Proposition~\ref{prop:mainn}. By
Lemma~\ref{lem:Q_{M,N}} we know that we need $\ll_M e^{O(\frac{\log \log M}{\log M})N}$ many of these. Moreover, by our assumption on $X_1$ we only need to look at sets $\mathcal V \subset
[0,N-1]$ with $|\mathcal V| \geq (\mu(X_{\geq M })-\epsilon) N$.
On the other hand, Proposition~\ref{prop:mainn}
gives that each of those sets $Q(\mathcal V )$ can be covered by
$\ll_M  e^{O(\frac{\log \log M}{\log M})N}e^{h_{\max}(\T)(N-\frac{1}{2}|\mathcal V|)}$ Bowen $N$-balls. 
Together we see that $Z$ can be covered by
 $$\ll_{M} e^{O(\frac{\log \log M}{\log M})N} e^{h_{\max}(\T)(N-\frac{1}{2}|\mathcal V|)}$$  Bowen
 $N$-balls. Applying Lemma~\ref{lem:entropy} we arrive at
 \begin{eqnarray*}
  h_{\mu}(\T)&\leq
&\liminf_{N \to
\infty}\frac{\log BC(N,\epsilon)}{N} \\
&\le& h_{\max}(\T)\left(1-\frac{(\mu(X_{\geq M })-\epsilon)}{2}\right)+O\left(\frac{\log  \log M}{\log M }\right).
\end{eqnarray*}
Since $\epsilon>0$ was arbitrary, we get that
\begin{equation}\label{eqn:mainthm}
h_{\mu}(\T)\leq h_{\max}(\T)\left(1-\frac{\mu(X_{\geq M })}{2}\right)+O\left(\frac{\log  \log M}{\log M }\right)
\end{equation}
which can be rewritten as
\begin{equation*}
\mu(X_{< M}) \ge 1-\frac{2}{h_{\max}(\T)} (h_{\max}(T)-h_{\mu}(T))+\phi(M)
\end{equation*}
where $\phi(M)=O\left(\frac{\log  \log M}{\log M }\right)$.

In the next we will show that the theorem holds for any $M>M_0:=\max\{e^{3h_{\max}(\T)},100\}$. Clearly the theorem holds if $\mu(X_{<M_0})>0$ so that we may assume $\mu(X_{<M_0})=0.$ Let us define the number $M_\mu$ by
$$M_\mu:=\inf\{M> M_0: \mu(X_{<M})>0\}.$$
The above argument implies that \eqref{eqn:mainthm} holds for any $M>M_\mu$. If $\mu(X_{<M_\mu})>0$ then \eqref{eqn:mainthm} also holds for $M=M_\mu.$ Otherwise if $\mu(X_{<M_\mu})=0$ then 
$$\lim_{n\to \infty}\mu(X_{\ge M_\mu+\frac{1}{n}})=\mu(X_{> M_\mu})=\mu(X_{\ge M_\mu})=1.$$
Now, using \eqref{eqn:mainthm} for $M+1/n$ instead of $M$ and taking the limit as $n \to \infty$ we get \eqref{eqn:mainthm} for $M=M_\mu.$ 
For any $M \in [M_0, M_\mu)$ we need to prove that \eqref{eqn:mainthm} holds. Since $\mu(X_{\ge M})=1$ for $M \le M_\mu$ we see that \eqref{eqn:mainthm} simplifies to
$$h_{\mu}(\T)\leq \frac{h_{\max}(\T)}{2}+O\left(\frac{\log  \log M}{\log M }\right).$$
Since $\frac{\log  \log M}{\log M }$ is decreasing for $M\ge 100$ and since the above equation holds for $M=M_\mu$ it clearly holds for any $M \in [M_0, M_\mu).$

For any $M> \max\{e^{3h_{\max}(\T)},100\}$, one can approximate the characteristic function of $X_{<M}$ by continuous functions with compact support and use \eqref{eqn:mainthm} to obtain the last part of the theorem.
\end{proof}

\section{Partitions}\label{sec:partitions}
For given $M,N\ge 1$ we recall the partition $Q_{M,N}:=\bigvee _{n=0}^{N-1}\T^{-n}\{X_{<M},X_{\ge M}\}.$
In this section we estimate the upper bound for the cardinality of $Q_{M,N}$ to prove Lemma~\ref{lem:Q_{M,N}}. Later we consider the refinement $P_{M,N}$ of the original partition $Q_{M,N}$ which is crucial in obtaining Proposition~\ref{prop:main}.

From now on, for simplicity, we assume that $a_j \ge 0$ for any $j \in [1,r+s].$ This in particular implies that the unstable subgroup $U^+$ is a subgroup of {\it lower} unipotent matrices in $G$. Also, with this assumption a component vector $(v_j',v_j'')$ under the iterations of $\T$ is getting short, that is $|(v_j',v_j'')| >|(v_j' e^{a_j/2} ,v_j'' e^{-a_j/2}) | ,$ means that $|v_j'| > |v_j''| e^{a_j/2}$ as otherwise if $a_j<0$ then we would get $|v_j''| > |v_j'| e^{a_j/2}$. Hence, the assumption $a_j \ge 0$ is simply a matter of ordering the coordinates of component vectors. 
\subsection{Proof of Lemma~\ref{lem:Q_{M,N}}}\label{sec:proofoflemma}
For any $x$, the partition element of $Q_{M,N}$ containing $x$ describes the time moments in $[0,N-1]$ for which $x$ stays above height $M$ (and hence when it is below height $M$) under the action of $\T$. So, we need to calculate the possible configurations of times in $[0,N-1]$. Our main tool to calculate the upper bound for the possible configurations is Lemma~\ref{lem:shortvector}. If there is a time when a point $x$ (under the action of $\T$) is above height $M$ then there is a considerable gap until the next time (if any) when $x$ reaches height $M$ again. This is because the vectors in $x$ can get short (under the action of $\T$) at most once and for another vector in $x$  to become short the earlier vector has to become of norm 1 at least. Now, we explicate the above discussion.
Assume that for a vector $v \in (\R^2)^r\times(\C^2)^s$ we have $\|v\|=\prod_{j=1}^{r+s}|(v_j',v_j'')|^{\delta_j}>1.$ We would like to know an estimate for the smallest possible time $n$ for which the vector $v$ reaches the norm $\le 1/M$ under the action of $\T$. It is easy to see that the best possible $n$ occurs for example when $v_j'=0$ for $j=1,...,r+s$. In this case, at time $n$ we must have 
$$\prod_{j=1}^{r+s}|(0,v_j''e^{-(i\theta_j+a_j/2)n})|^{\delta_j}\le\frac{1}{M}.$$
Since $\|v\|> 1$, we must have
$e^{-\frac{n}{2}\sum_{j=1}^{r+s} a_j \delta_j}<  1/M \text{ which gives } $
$$n \ge \frac{2\log M}{h_r+2h_s}.$$ 
Similarly, for a vector of norm at most $1/M$, under the action of $\T$,  the smallest possible time moment when the norm becomes greater than $1$ is again $\ge \frac{2\log M}{h_r+2h_s}$. We also note that for any vector $v$ in $x$ if the sequence $(\|\T^n(v)\|)_{n \ge 0}$ gets increased at some time then it becomes monotone increasing from that time moment. Thus, in a time interval of length $2\lfloor \frac{2\log M}{h_r+2h_s}\rfloor$, for any point $x$ in $X$ there can be at most one time interval on which $x$ stays above height $M$. Hence, $Q_{M,\lfloor\frac{2\log M}{h_r+2h_s}\rfloor}$ has at most $\left(\begin{array}{c}2\lfloor\frac{2\log M}{h_r+2h_s}\rfloor\\2\end{array}\right)\ll \log^2 M$ many elements. On the other hand, to obtain $Q_{M,N}$ we need to take refinements of $\lfloor \frac{N}{2\lfloor\frac{2\log M}{h_r+2h_s}\rfloor-1}\rfloor$ many pre-images of $Q_{M,\lfloor\frac{2\log M}{h_r+2h_s}\rfloor}$ and at most $2\lfloor\frac{2\log M}{h_r+2h_s}\rfloor-1$ many of $\{X_{<M},X_{\ge M}\}.$ For $M \ge e^{h_r+2h_s}$ we have 
$$\left\lfloor \frac{N}{2\lfloor\frac{2\log M}{h_r+2h_s}\rfloor-1}\right\rfloor<\frac{N}{\frac{4\log M}{h_r+2h_s}-3}\le\frac{N(h_r+2h_s)}{\log M}$$
Hence, we obtain that the cardinality of $Q_{M,N}$ is
$$\ll (\log^2 M)^{\frac{N(h_r+2h_s)}{\log M}}\le e^{\frac{2(h_r+2h_s)\log \log M}{\log M}N}.$$

\qed
\subsection{The refined partition $P_{M,N}$} We now consider the refinement $P_{M,N}$ of $Q_{M,N}$. It is a bit technical and the reason why this refinement is needed comes from the product structure of the space $G$ and in particular the way we define the height function $\h(\cdot).$ The partition elements of $Q_{M,N}$ give information as when the trajectory of a point under $\T$ goes into the cusp and when it comes back. Due to the way the height $\h(\cdot)$ is defined this does not provide much information on individual components $(v_j',v_j'')$ of the short vectors $v \in (\R^2)^r \times (\C^2)^s$ even if we know that $\|\T^n v\|$ is decreasing on some time interval in $[0,N-1]$. Thus, what we really need is a partitioning of the space $X$ which describes whether components of short vectors under iterates of $\T$ decreases or increases. On the other hand, if a component vector gets shorter in $n$ iterates under $\T$, that is, if $|\T^n(v_j',v_j'')|=|(v_j'e^{n a_j/2},v_j'' e^{-n a_j/2})|\le|(v_j',v_j'')|$, then it is easy to see that we must have $|v_j'| e^{n a_j/2}\le |v_j''|$. This simple observation hints the importance of knowing the ratios $\frac{|v_j'|}{|v_j''|}$ of component vectors. Thus, elements of our new partition $P_{M,N}$ should describe these ratios (cf. \eqref{eqn:Q(J_m)}) of component vectors of short vectors as we define now.

Our goal is to refine the partition $Q_{M,N}$ further by partitioning most of its elements. Let $Q$ be one of its elements. Then there exists $\mathcal{V} \subset [0, N-1]$ such that 
\begin{multline}
\label{eqn:Q(V)}
Q:=Q(\mathcal V)\\
=\{x \in X : \text{ for all } n \in [0,N-1], \T^n(x)\in X_{\ge M} \text{ if and only if } n \in \mathcal V\}.
\end{multline}
We split $\mathcal V$ into maximal intervals $V^{\mathcal V}_1,\dots, V^{\mathcal  V}_k$ for some $k\in \N$. For $m=1,2,\dots,k$ we write $V^{\mathcal V}_m=[b^{\mathcal  V}_m,b^{\mathcal  V}_m+\ell^{\mathcal V}_m]$. 

 For any $j \in [1,r+s]$ recall the fixed number $a_j$ appeared in the definition of $\T$. For each $j \in [1, r+s]$ and $m \in\{1,2,\dots,k\}$ let us decompose the extended reals into the following $\ell_m^{\mathcal V}+2$ subintervals: 
\begin{align}
\label{eqn:I^{(m,j)}}
I_{0,j}(V^{\mathcal V}_m)&=[-\infty,b_m^{\mathcal V}],\,I_{\ell_m^{\mathcal V}+1,j}(V^{\mathcal V}_m)=(b_m^{\mathcal V}+\ell_m^{\mathcal V} a_j ,\infty],\\
 I_{n,j}(V^{\mathcal V}_m)&=(b_m^{\mathcal V}+(n-1)a_j,b_m^{\mathcal V}+na_j] \text{ for } n\in [1,\ell_m^{\mathcal V}].
\end{align}
We write 
$$\mathcal{I}_j(V^{\mathcal V}_m)=\{I_{n,j}(V^{\mathcal V}_m): n\in [0,\ell_m^{\mathcal V}+1]\} \text{ for } m\in\{1,2,\dots,k\} \text{ and } j\in [1,r+s].$$ 
We first note that for any $x \in Q$ there exists a unique primitive vector $v \in \T^{b_m^{\mathcal V}-1}(x)$ such that 
\begin{equation}
\label{eqn:vm}
\|\T^n(v)\|\le \frac{1}{M} \text{ for } n \in [1,\ell_m^{\mathcal V}+1].
\end{equation}
We fix $m \in \{1,2,\dots, k\}$ and for each $j \in [1,r+s]$ we pick one interval $J_j(V_m^{\mathcal V})$ from the set $\mathcal{I}_j(V^{\mathcal V}_m)$ and consider the product set
\begin{equation}
\label{eqn:J_m}
J(V_m^{\mathcal V})=J_1(V_m^{\mathcal V})\times \cdots \times J_{r+s}(V_m^{\mathcal V}) .
\end{equation}
 Now, for any such product set $J(V_m^{\mathcal V})$ we associate a partition element, which could be empty, in $Q$ given by
\begin{multline}
\label{eqn:Q(J_m)}
Q(J(V_m^{\mathcal V}) ):=\{x \in Q : \exists v\in\T^{b_m^{\mathcal V}-1}(x) \text{ such that } \eqref{eqn:vm} \text{ holds and }\\
  |v_j''|=|v_j'| e^{s_j} \text{ for some } s_j \in J_j(V_m^{\mathcal V})-b_m^{\mathcal V}\}.
\end{multline}
For any $m\in \{1,2,\dots,k\}$ we fix one partition element $Q(J(V_m^{\mathcal V}))$ as in \eqref{eqn:Q(J_m)} and define the following further refined partition element 
\begin{equation}
\label{eqn:P(V)}
P(\mathcal V)=\bigcap_{m=1}^{k} Q(J(V_m^{\mathcal V})).
\end{equation}
In this way, for any choice of $Q \in Q_{M,N}$ and any choice of $J(V_m^{\mathcal V})$ as in \eqref{eqn:J_m} we obtain one partition element which is contained in $Q$. The collection of all possible $P(\mathcal V)$ as in \eqref{eqn:P(V)} gives a refined partition $P_{M,N}$ of $Q_{M,N}.$ 

For further motivation why the partition $P_{M,N}$ is crucial we refer to \S~\ref{sec:res}, in particular see Lemma~\ref{lem:res}. 
\begin{lem}
\label{lem:P_{M,N}}
For $M > e^{3h_{\max}(\T)}$ and $N \in \N$ the cardinality of the partition $P_{M,N}$ constructed above is $\ll e^{O(\frac{\log \log M}{\log M})N}$ where the implied constants are independent of $M,N.$
\end{lem}
\begin{proof}
Consider a partition element $Q(\mathcal V)$ of $Q_{M,N}$ as in \eqref{eqn:Q(V)}.
Let $Q(J(V_m^{\mathcal V}))$ be as in \eqref{eqn:Q(J_m)} and $P(\mathcal V)$ be as in \eqref{eqn:P(V)}. There are at most $(|V_m^{\mathcal V}|+2)^{r+s}$ possible ways to choose $J(V_m^{\mathcal V})$ and hence $(|V_m^{\mathcal V}|+2)^{r+s}$ possible ways to choose $Q(J(V_m^{\mathcal V}))$ for a fixed $m \in [1,k]$. Thus, the number of partition elements of $P_{M,N}$ contained in $Q(\mathcal V)$ is \begin{multline*}
(|V_1^{\mathcal V}|+2)^{r+s}(|V_2^{\mathcal V}|+2)^{r+s}\cdots (|V_k^{\mathcal V}|+2)^{r+s}\\
=\exp\left((r+s)[\log (|V_1^{\mathcal V}|+2)+\log (|V_2^{\mathcal V}|+2)+\cdots+\log (|V_k^{\mathcal V}|+2)]\right).
\end{multline*}
This is $$\ll \exp ((r+s)\log (|V_1^{\mathcal V}| |V_2^{\mathcal V}|...|V_k^{\mathcal V}|)).$$
We have $$|V_1^{\mathcal V}| |V_2^{\mathcal V}|...|V_k^{\mathcal V}|\le \left(\frac{|V_1^{\mathcal V}| + |V_2^{\mathcal V}|+...+|V_k^{\mathcal V}|}{k}\right)^k\le\left(\frac{N}{k}\right)^k.$$
Also, note that for the function $f(x)=(\frac{N}{x})^x=(N)^x e^{-x\log
 x}$ its derivative 
 \begin{align*}
 f'(x)&=(N)^x \log(N) e^{-x\log x}+(N)^x
 e^{-x\log x}(-\log x -1)\\
 &=(N)^x e^{-x\log x}(\log (N)-\log x-1) .
 \end{align*}
  Hence
 $f(x)=(\frac{N}{x})^x$ is increasing on $[1,\frac{N}{e}]$. On the other hand, from the proof of Lemma~\ref{lem:Q_{M,N}} we know that 
 $$k \le \left \lceil\frac{N}{2\lfloor \frac{2\log M}{h_r+2h_s}\rfloor}\right \rceil \le \left\lceil\frac{(h_r+2h_s)N}{2\log M}\right\rceil \le \max\left\{1, \frac{(h_r+2h_s)N}{\log M}\right\}.$$
 If $k=1$ then $(\frac{N}{k})^k=N$. Otherwise, $k \le \frac{(h_r+2h_s)N}{\log M}$ and for $M \ge e^{e(h_r+2h_s)}$ we have
 $$\left(\frac{N}{k}\right)^k\le\left(\frac{N}{\frac{(h_r+2h_s)N}{\log M} }\right)^\frac{(h_r+2h_s)N}{\log M}=\left(\frac{\log M}{h_r+2h_s}\right)^\frac{(h_r+2h_s)N}{\log M}. $$
 Hence, the number of partition elements of $P_{M,N}$ contained in $Q(\mathcal V)$ is $\ll e^{(r+s)\log (N)}$ if $k=1$ and otherwise it is
 $$\ll \exp\left((r+s)\log \left(\left(\frac{\log M}{r+s}\right)^\frac{(h_r+2h_s)N}{\log M}\right)\right)\ll e^{O(\frac{\log \log M}{\log M})N}. $$
 In either case, the number of partition elements of $P_{M,N}$ contained in $Q(\mathcal V)$ is $\ll e^{O(\frac{\log \log M}{\log M})N}$. Thus, together with Lemma~\ref{lem:Q_{M,N}} we deduce that $P_{M,N}$ has $\ll e^{O(\frac{\log \log M}{\log M})N}$ elements for $M > e^{3 h_{\max}(\T)}$.
\end{proof}
\section{Main proposition}
\label{sec:mainprop}
 In this section we calculate the number of Bowen $N$-balls needed to cover each partition element of $P_{M,N}$. We recall that a  Bowen $N$-ball is a translate of $B_N=\bigcap_{n=0}^{N-1}a^{n}B_{\eta}^{G} a^{-n}$ in $X$. We note that the Bowen balls are balls in a different metric that induces the same topology.

Let $M,N \ge 1$ be given. Let $P(\mathcal V)$ be a partition element of $P_{M,N}$ as in \eqref{eqn:P(V)} such that $P(\mathcal V)\subset X_{<M}$. We recall that by definition $\mathcal V $ is a subset of $ [0,N-1]$ and for all $n \in [0,N-1]$ we have that $\T^n(x) \in X_{\ge M}$ if and only if $n \in \mathcal V.$ In particular, the additional restrictive assumption above is equivalent to $\mathcal V$ being in $ (0,N-1]$. 
\begin{prop}
\label{prop:main}
The partition element $P(\mathcal V) \in P_{M,N}$ with $P(\mathcal V) \subset X_{<M}$ can be covered by 
$$\ll_M c_0^{\frac{h_{\max}(\T)}{\log M}N}e^{h_{\max}(\T)(N-\frac{|\mathcal V|}{2})}$$
 Bowen $N$-balls for some universal constant $c_0 \ge 1$. 
\end{prop}
\begin{proof}[Proof of Proposition~\ref{prop:mainn}]
We note that we partitioned any element $Q(\mathcal V) \in Q_{M,N}$ into elements $P(\mathcal V)$ of $P_{M,N}.$ From Lemma~\ref{lem:P_{M,N}} we know that there are at most $\ll e^{O(\frac{\log \log M}{\log M})N}$ such elements of $P_{M,N}$ for $M>e^{3h_{\max}(\T)}$. On the other hand, using Proposition~\ref{prop:main} we deduce that each such element $P(\mathcal V)$ can be covered by $\ll_M c_0^{\frac{h_{\max}(\T)}{\log M}N}e^{h_{\max}(\T)(N-\frac{|\mathcal V|}{2})}$ Bowen $N$-balls. By enlarging the implicit constant in $O(\cdot)$ we may assume that $c_0^{\frac{h_{\max}(\T)}{\log M}N} \le e^{O(\frac{\log \log M}{\log M})N}$. Thus, we conclude that any partition element $Q(\mathcal V)$ with $\mathcal V \in (0,N-1]$ can be covered by $\ll_M e^{O(\frac{\log \log M}{\log M})N} e^{h_{\max}(\T)(N-\frac{|\mathcal V|}{2})}$ Bowen $N$-balls which completes the proof.
\end{proof}
We now return to the statement of Proposition~\ref{prop:main}. Roughly, we note that since the number of elements of $P_{M,N}$ is slow exponential as $N \to \infty $, to calculate the entropy it is sufficient to consider the covers of each partition element $P_{M,N}$ by Bowen balls. Since we only need to count the number of covers of most of the space $X$ (cf. Lemma~\ref{lem:entropy}) it is reasonable to consider only the partitions $P(\mathcal V) \in P_{M,N}$ with $P(\mathcal V) \subset X_{<M}$. It is not hard to show that each such partition element $P(\mathcal V)$ can be covered by $\ll e^{h_{\max}(\T)N}$ Bowen $N$-balls. Thus, the significant factor in Proposition~\ref{prop:main} is $e^{-\frac{h_{\max}(\T)}{2}|\mathcal V|}$. Before we start proving Proposition~\ref{prop:main} we need some preliminary preparations. 
\subsection{Restrictions of perturbations}
\label{sec:res}
If there are two points in $X_{<M}$ which are $\eta$-close to each other such that they both stay above height $M$ for some time interval, then we would like to say that these points must be even closer to each other in the unstable direction $U^+$. This is not true in general. However, if additionally we know that they are in the same partition element of $P_{M,N}$ then we will show that this is indeed the case. 

As before let $U^+, \,U^-,\, L$ be the unstable, stable, and centralizer subgroups of $G$ w.r.t. $a$ respectively. We naturally embed $U^+$ into $\R^r \times \C^s.$ We let $u^+({\bf t}) \in U^+$ be the element that corresponds to ${\bf t}=(t_1,t_2,\dots,t_{r+s}) \in \R^r\times \C^s.$
For the rest of the section we fix one $P(\mathcal V) \in P_{M,N}$ as in \eqref{eqn:P(V)}. Recall that $V_m^{\mathcal V}=[b_m^{\mathcal V},b_m^{\mathcal V}+\ell_m^{\mathcal V}], m=1,2,\dots,k$ are the maximal intervals such that $\mathcal V =\cup_{m=1}^k V_m^{\mathcal V}$. We fix $V_m^{\mathcal V}$ for some $m=1,\dots,k$ and for simplicity we denote $V_m^{\mathcal V}=[b,b+\ell]$. From \eqref{eqn:P(V)} we know that $P(\mathcal V)=\bigcap_{m=1}^{k} Q(J(V_m^{\mathcal V}))$ for some $Q(J(V_m^{\mathcal V}))$ as in \eqref{eqn:Q(J_m)}, namely
\begin{multline}
\label{eqn:Q(J_m)'}
Q(J(V_m^{\mathcal V}) ):=\{x \in Q : \exists v\in\T^{b-1}(x) \text{ such that } \eqref{eqn:vm} \text{ holds and }\\
  |v_j''|=|v_j'| e^{s_j} \text{ for some } s_j \in J_j(V_m^{\mathcal V})-b\}.
\end{multline}
\begin{lem}
\label{lem:res}
Let $x,y \in P(\mathcal V) \cap \T^{N-1}(X_{<M})$ with $\T^{b-1}(y)= \T^{b-1}(x) u^+({\bf t})g$ for some $u^+({\bf t})\in B_{\eta/2}^{U^+}$ and $g\in B_{\eta/2}^{U^-L}$. Then for any $j \in \{1,2,\dots,r+s\}$ we have $|t_j| \ll e^{b-n_j}$ where $n_j$ is the left end point of the interval $J_j(V_m^{\mathcal V})$.
\end{lem}
\begin{proof}
If $J_j(V_m^{\mathcal V})=[-\infty,b]=I_{0,j}(V_m^{\mathcal V})$ then $n_j=-\infty$ and in this case the lemma is trivial. So, we may assume $J_j(V_m^{\mathcal V})\neq I_{0,j}(V_m^{\mathcal V})$ so that $n_j \ge b.$

 By maximality of $V_m^{\mathcal V}$ and the fact that $\mathcal V \subset (0,N-1]$ we know that 
 \begin{align*}
 \T^{b-1}(x),\T^{b-1}(y) &\in X_{<M} \text{ and }\\ 
 \T^n(\T^{b-1}(x)),\T^n(\T^{b-1}(y)) &\in X_{\ge M} \text{ for any } n \in [1,l+1].
\end{align*}
Thus there exist vectors $v \in \T^{b-1}(x)$ and $w\in \T^{b-1}(y)$ such that \eqref{eqn:vm} holds, that is
\begin{equation}
\label{eqn:vmw}
\|\T^n(v)\|,\|\T^n(w)\|\le 1/M  \text{ for } n \in [1,\ell+1].
\end{equation}
On the other hand, from \eqref{eqn:Q(J_m)'} for $v,w$ in the standard notation we know that
  $$|v_j''|=|v_j'|e^{s_j} \text{ and } |w_j''|=|w_j'|e^{r_j} \text{ for some } s_j,r_j \in J_j(V_m^{\mathcal V})-b.$$
We note that $v_j'' \neq 0 \neq w_j''$ since $(v_j',v_j''),(w_j',w_j'')\neq (0,0)$ (they are rows of matrices of determinant equal to 1) and $s_j,r_j \ge 0$. In particular, if $n_j$ is the left end point of the interval $J_j(V_m^{\mathcal V})$ then we have
 \begin{equation}
 \label{eqn:vw}
 \frac{|v_j'|}{|v_j''|} \le e^{b-n_j} \text{ and } \frac{|w_j'|}{|w_j''|} \le e^{b-n_j}. \end{equation}
 Also, we know that $w=vu^+({\bf t})g.$ So, for $g=(g_1,\dots,g_{r+s})$ we have $(w_j',w_j'')=(v_j',v_j'')\left(\begin{array}{cc} 1&0\\t_j &1 \end{array}\right)g_j=(v_j'+t_jv_j'',v_j'')g_j$ (under the assumption that $a_j \ge 0$ where $a_j$ is as in the definition of $a$). For $g_j=\left(\begin{array}{cc} d&u\\0 &1/d \end{array}\right)$ we obtain that
$$(w_j',w_j'')=(d(v_j'+t_jv_j''),u(v_j'+t_jv_j'')+v_j''/d).$$
 Now from \eqref{eqn:vw} we get
$$e^{b-n_j}\ge \frac{|w_j'|}{|w_j''|}=\frac{|d(v_j'+t_jv_j'')|}{|u(v_j'+t_jv_j'')+v_j''/d|} \gg \frac{|v_j'+t_jv_j''|}{|v_j''|}=\left |\frac{v_j'}{v_j''}+t_j\right |$$
since $d$ is close to $1$ and $u$ is close to $0$. Together with \eqref{eqn:vw} we deduce that
$$|t_j| \ll e^{b-n_j}.$$
\end{proof}
Lemma~\ref{lem:res} alone does not tell us if $x,y$ should be even closer to each other in the unstable direction since for example $n_j$ could be equal to $b$. Even if $n_j>b$ we still do not know an effective lower bound for $n_j$. This is because we have only considered one part of the defining properties of $Q(J(V_m^{\mathcal V}))$.  We have not considered the fact that $x,y$ stay above height $M$ in $[1,\ell+1].$ In the next lemma we use this fact to obtain the relation among the intervals $J_j(V_m^{\mathcal V})$.
\begin{lem}
\label{lem:relation}
Let $J(V_m^{\mathcal V})$ be as in \eqref{eqn:J_m} and consider $x \in Q(J(V_m^{\mathcal V}))$ with $v \in \T^{b-1}(x)$ as in \eqref{eqn:Q(J_m)'}. Let $S=\{s_1,...,s_{r+s}\}$ and  $\{i_1,...,i_L\}$ be the subset of $S$ which are $\le 0$, let $j_1,...,j_C$ be the subset of $S$ such that $s_{j_i} \in (0,(\ell+1)a_{j_i})$, and let $k_1,...,k_R$ be the subset of $S$ such that $s_{k_i} > (\ell+1)a_{k_i}$. In particular, $L+C+R=r+s.$ Then
$$(\ell +1)\left(\sum_{n=1}^L( a_{i_n} k_{i_n})+\sum_{n=1}^C (a_{j_n}k_{j_n})-\sum_{n=1}^R (a_{k_n}k_{k_n})\right)<2\sum_{n=1}^C (s_{j_n}k_{j_n}). $$
\end{lem}
\begin{proof}
Let us consider the $j$-th component vector $(v_j',v_j'')$ of $v$. $\T$ acts on $v$ and hence it acts on each of its components and we have 
$$\T^n((v_j',v_j''))=(v_j'e^{in\theta_j}e^{na_j/2},v_j''e^{-in\theta_j}e^{-na_j/2})$$
where as before $\theta_j=0$ if $j\le r$, and $a_j \ge 0$ for any $j \in [1,r+s].$Thus,
\begin{equation*}
|\T^n((v_j',v_j''))|=\max\{|v_j'e^{na_j/2}|,|v_j''e^{-na_j/2}|\}=\begin{cases} |v_j''|e^{-na_j/2} & \text{if } na_j < s_j \\ |v_j'|e^{na_j/2} & \text{if } na_j \ge s_j
\end{cases}
\end{equation*}
since $|v_j'|e^{s_j/2}=|v_j''|e^{-s_j/2}$. We also note that 
$$|(v_j',v_j'')|=\begin{cases} |v_j'|& \text{if }  s_j \le 0  \\ |v_j''|& \text{if } s_j> 0\end{cases}.$$
 Together we get
\begin{equation}
\label{eqn:T^n}
\frac{|\T^{\ell+1}((v_j',v_j''))|}{|(v_j',v_j'')|}=\begin{cases} e^{\frac{(\ell+1)a_j}{2}} & \text{if }  s_j \le 0  \\ e^{\frac{(\ell+1)a_j}{2}-s_j} &\text{if } s_j \in (0,(\ell+1)a_j]\\ e^{-\frac{(\ell+1)a_j}{2}} & \text{if } s_j>(\ell+1)a_j.
\end{cases}
\end{equation}
By the assumption \eqref{eqn:vm} on the vector $v \in (\R^2)^r\times (\C^2)^s$ we have
$$\|v\| > \frac{1}{M} \text{ and }\|\T^n(v)\|\le \frac{1}{M}  \text{ for } n\in [1,\ell+1].$$
In particular, this gives
\begin{equation}
\label{eqn:res}
\frac{\|\T^{\ell+1}(v)\|}{ \|v\|} < 1.
\end{equation}
Now, from \eqref{eqn:T^n} and \eqref{eqn:res} we get
\begin{multline*}
\frac{\prod_{j=1}^{r+s}|\T^{\ell+1}((v_j',v_j''))|^{\delta_j}}{\prod_{j=1}^{r+s}|(v_j',v_j'')|^{\delta_j}}\\
=\exp\left(\frac{\ell+1}{2}\sum_{n=1}^L( a_{i_n} k_{i_n})+\frac{\ell+1}{2}\sum_{n=1}^C (a_{j_n}k_{j_n})-\sum_{n=1}^C (s_{j_n}k_{j_n})-\frac{\ell+1}{2}\sum_{n=1}^R (a_{k_n}k_{k_n})\right)<1.
\end{multline*}
The exponent simplifies to
$$(\ell +1)\left(\sum_{n=1}^L( a_{i_n} k_{i_n})+\sum_{n=1}^C (a_{j_n}k_{j_n})-\sum_{n=1}^R (a_{k_n}k_{k_n})\right)<2\sum_{n=1}^C (s_{j_n}k_{j_n}). $$
\end{proof}
The next lemma shows how we apply the above two lemmas. The reader can skip the lemma and come back when it is mentioned in the proof of Proposition~\ref{prop:main}. Recall the embedding  of $U^+$ into $\R^r\times\C^s.$
\begin{lem}
\label{lem:vol}
 Let $V_m^{\mathcal V}=[b,b+\ell]$ and $Q(J(V_m^{\mathcal V}))$ be as before and let $B',B''$ be given positive constants. Let us consider the set $D:=\{u({\bf t}) \in U^+ : |t_j|< B'\min\{\eta, e^{b-n_j}\}, j=1,\dots,r+s\}$ where $n_j$ is the left end point of the interval $J_j(V_m^{\mathcal V})$. Then the set $D$ can be decomposed into 
 $$\ll e^{\frac{h_{\max}(\T)}{2}\ell}$$
  disjoint sets of the form $E:=\{u({\bf t}) \in U^+ : |t_j|<B''\eta e^{-\ell a_j}, j=1,\dots,r+s\}$.
\end{lem}
It is easy to see  that the set $E$ is roughly the unstable part of a Bowen $\ell$-ball which hints the relation to Proposition~\ref{prop:main}. One can see that a unit ball in $U^+$ can be covered by $\ll e^{h_{\max}(\T)\ell}$ translates of the set $E$. However, from the lemma above we see that if we consider a subset of a unit ball in $U^+$ whose elements stay above height $M$ under $\T$ on $[b,b+\ell]$ and moreover if the elements of this subset behave similarly, that is, if the set is in the partition $Q(J(V_m^{\mathcal V}))$ then we have save $h_{\max}(\T)\ell/2$ in the exponent.
\begin{proof}
For any $j=1,2,\dots,r+s$ let us consider the ball around 0 of radius $B'\min\{\eta, e^{b-n_j}\}$ in $\R$ or in $\C$ depending whether $j \le r $ or not and decompose it into the small balls of radius $B''\eta e^{-\ell}$. If $n_j < b$ (in which case $n_j=-\infty$) then there are $\ll e^{\ell a_j}$ small subintervals if $j\le r$ and there are $\ll e^{2\ell a_j}$ small balls if $j>r$. Suppose $n_j \ge b$. If $j \le r$ then there are $\ll e^{\ell a_j+b-n_j}$ small subintervals and if $j >r$ then there are $\ll e^{2(\ell a_j+b-n_j)}$ small balls. We note that if $n_j \ge b+\ell a_j$ (in which case $n_j=b+\ell a_j$) then there are $\ll 1$ small subintervals or $\ll 1$ small balls depending on $j$. We have $i_1,...,i_L,j_1,...,j_C,k_1,...,k_R$ as in Lemma~\ref{lem:relation}. Now, let $i_1',...,i_{L'}'$ be the subset of $\{i_1,...,i_L\}$ which are $\le r$ and $i_1'',...,i_{L''}''$ be the rest. Similarly, we consider the subsets $j_1',...,j_{C'}'$ and $j_1'',...,j_{C''}''$ of $j_1,...,j_C$.

Therefore, the set $D$ contains at most 
\begin{align*}
\ll& 1^R \exp\left(\ell\sum_{n=1}^{L'}a_{i_n'} +2\ell\sum_{n=1}^{L''}a_{i_n''}+\ell\sum_{n=1}^{C'}a_{j_n'}+bC'-\sum_{n=1}^{C'} n_{j_n'} +2(\ell\sum_{n=1}^{C''}a_{j_n''}+bC''-\sum_{n=1}^{C''} n_{j_n''})\right)\\
=&\exp\left(\ell\left(\sum_{n=1}^{L'}a_{i_n'} +2\sum_{n=1}^{L''}a_{i_n''}+\sum_{n=1}^{C'}a_{j_n'}+2\sum_{n=1}^{C''}a_{j_n''}\right)+b(C'+2C'')-\sum_{n=1}^{C'} n_{j_n'} -2\sum_{n=1}^{C''} n_{j_n''})\right)\\
=&\exp\left(\ell\left(\sum_{n=1}^{L}(a_{i_n}k_{i_n}) +\sum_{n=1}^{C}(a_{j_n}k_{j_n})\right)+b\sum_{n=1}^C k_{j_n}-\sum_{n=1}^{C} (n_{j_n}k_{j_n}))\right)
\end{align*}
 many disjoint sets of the form  $E$.
 
On the other hand, Lemma~\ref{lem:relation} gives
$$(\ell +1)\left(\sum_{n=1}^L( a_{i_n} k_{i_n})+\sum_{n=1}^C (a_{j_n}k_{j_n})-\sum_{n=1}^R (a_{k_n}k_{k_n})\right)<2\sum_{n=1}^C (s_{j_n}k_{j_n}) $$ 
where $s_{j_k} \in J_{j_k}(V_m^{\mathcal V})-b=(n_{j_k}-b,n_{j_k}+a_{j_k}-b]$. Thus,
$$(\ell +1)\left(\sum_{n=1}^L( a_{i_n} k_{i_n})+\sum_{n=1}^C (a_{j_n}k_{j_n})-\sum_{n=1}^R (a_{k_n}k_{k_n})\right)<2\sum_{n=1}^C ((n_{j_n}+a_{j_n}-b)k_{j_n}) $$ 
and since $\sum_{n=1}^L( a_{i_n} k_{i_n})+\sum_{n=1}^C (a_{j_n}k_{j_n})+\sum_{n=1}^R (a_{k_n}k_{k_n})=h_r+2h_s=h_{\max}(\T)$ we obtain
$$(\ell +1)\left(2\sum_{n=1}^L( a_{i_n} k_{i_n})+2\sum_{n=1}^C (a_{j_n}k_{j_n})-h_{\max}(\T)\right)<2\sum_{n=1}^C ((n_{j_n}+a_{j_n}-b)k_{j_n}). $$ 
Dividing both sides by $2$ and simplifying the expression we get
\begin{multline*}
\ell\left(\sum_{n=1}^L( a_{i_n} k_{i_n})+\sum_{n=1}^C (a_{j_n}k_{j_n})\right)+b \sum_{n=1}^C k_{j_n}-\sum_{n=1}^C (n_{j_n}k_{j_n})\\
<-\sum_{n=1}^L (a_{i_n} k_{i_n})+\frac{(\ell+1)h_{\max}(\T)}{2}.
\end{multline*}
Hence, the set $D$ can be decomposed into 
\begin{align*}
&\ll \exp\left( -\sum_{n=1}^L (a_{i_n} k_{i_n})+\frac{(\ell+1)h_{\max}(\T)}{2}\right) \ll \exp\left(\frac{ h_{\max}(\T)}{2}\ell\right)
\end{align*}
disjoint sets of the form $E$. 
\end{proof}
\subsection{The proof of Proposition~\ref{prop:main}}
Let $P(\mathcal V) \in P_{M,N}$ be given. Since $X_{<M}$ is pre-compact it suffices to restrict ourselves to a neighborhood $\mathcal{O}$ of some $x_0 \in X_{<M}\cap P(\mathcal V)$. We let $\mathcal{O}=x_0 B_{\eta/2}^{U^+}B_{\eta/2}^{U^-L }$ be a neighborhood of such $x_0 \in X_{<M}$ and define the set $P_{\mathcal O}(\mathcal V)$ by
$$P_{\mathcal O}(\mathcal V )= \mathcal O \cap P(\mathcal V).$$
It suffices to prove that the set $P_{\mathcal O}(\mathcal V)$ can be covered by $\ll c_0^{\frac{h_{\max}(\T)N}{\log M}}e^{h_{\max}(\T)(N-\frac{1}{2}|\mathcal V|)}$ Bowen $N$-balls for some universal constant $c_0 \ge 1$. 

Let us make some observations. If we consider the image of $\mathcal O$ under $\T^n$ we obtain the set
$$\T^n(\mathcal O)=\T^n(x_0) (a^{-n}B_{\eta/2}^{U^+}a^n) a^{-n} B_{\eta/2}^{U^-L} a^n.$$
We see that the $j$th component of the $U^+$-part gets stretched by the factor $e^{na_j}$. Here again we naturally embed $U^+$ into $\R^r\times \C^s$. Under this identification, dividing $(a^{-n}B_{\eta/2}^{U^+}a^n) $ into $\prod_{j=1}^{r+s}\lceil e^{na_j} \rceil^{\delta_j}$ many small parts we obtain the sets of the form
\begin{equation*}
\T^n(x_0) u^+B_{\eta/2}^{U^+} a^{-n} B_{\eta/2}^{U^-L} a^n
\end{equation*}
for some $u^+ \in U^+.$ Now, if we take the pre-image under $\T^n$ of these sets then we obtain the similar sets 
$$\T^{-n}(\T^n(x_0) u^+)a^n B_{\eta/2}^{U^+} a^{-n} B_{\eta/2}^{U^-L} $$
as before. It is not hard to see that the set $\T^{-n}(\T^n(x_0) u^+)a^n B_{\eta/2}^{U^+} a^{-n} B_{\eta/2}^{U^-L}$ is contained in the forward Bowen $n$-ball $\T^{-n}(\T^n(x_0) u^+) B_n^+$. This in particular shows that $\mathcal O$ can be covered by $\ll  e^{(h_r+2h_s)n}$ many forward Bowen $n$-balls which is the reason why the maximal entropy $h_{\max}(\T)$ is $h_r+2h_s$. However, using Lemma~\ref{lem:vol} we will show that we in fact need fewer Bowen balls to cover the set $\mathcal O$.

Let us recall that we decompose $\mathcal V$ into ordered maximal subintervals $V_m^{\mathcal V}$ so that we have
$$\mathcal V=V_1^{\mathcal V} \cup V_2^{\mathcal V} \cup ... \cup V_k^{\mathcal V}. $$
Now we let $[0,N-1] \setminus \mathcal V=W_1\cup W_2 \cup ... \cup W_{k'}$ where $W_m$ are again ordered maximal intervals. We inductively prove the following:

If $[0,b-1]=V_1^{\mathcal V}\cup V_2^{\mathcal V} \cup ... \cup V_{m-1}^{\mathcal V} \cup W_1 \cup W_2 \cup...\cup W_{n'}$ then for some constant $c_0$ the set $P_{\mathcal O}(\mathcal V)$ can be covered by 
$$\le c_0^{m-1+n'} \exp\left(h_{\max}(\T)\left[(b-1)-\frac{(|V_1^{\mathcal V}|+\cdots +|V_{m-1}^{\mathcal V}|)}{2}\right]\right)$$
 pre-images under $\T^{b-1}$ of sets of the form
\begin{equation}
\label{eqn:earlier}
\T^{b-1}(x_0)u^+B_{\eta/2}^{U^+}a^{-b+1}B_{\eta/2}^{U^-L}a^{b-1}.
\end{equation}
For the interval $[0,0]$ the claim is obvious. Now, assume that the claim is true for the interval $[0,b-1]$ as above. In the inductive step, if the next interval is $W_{n'+1}$ then once we divide each set obtained earlier into 
$$\prod_{j=1}^{r+s}\lceil e^{|W_{n'+1}|a_j} \rceil^{\delta_j}  \le c_0 e^{h_{\max}(\T)(|W_{n'+1}|)} $$
 small ones for some constant $c_0$, we just keep all of them. So, assume that the next interval is $V_{m}^{\mathcal V}=[b,b+\ell]$. Let $Y$ be one of the sets \eqref{eqn:earlier} obtained in the earlier step. We would like to estimate the upper bound to cover $Y$ by pre-images under $\T^{\ell}$ of sets of the form
\begin{equation}
\label{eqn:smallsets}
\T^{b-1+\ell}(x_0)u^+({\bf t})B_{\eta/2}^{U^+}a^{-b+1-\ell}B_{\eta/2}^{U^-L}a^{b-1+\ell}.
\end{equation}

We are interested in the points $x\in Y$ for which  $\T^{-b+1}(x)$ is in $Q(J(V_m^{\mathcal V}))$. We know by assumption that $x_0$ is one of them. If $x \in Y$ is another one then by Lemma~\ref{lem:res} there exists ${\bf t} \in B_{\eta/2}^{\R^r\times \C^s}$ such that $x=x_0u^+({\bf t})g$ for some $g \in B_{\eta/2}^{U^-L}$ and for $j \in [1,r+s]$, $|t_j| \ll e^{b-n_j}$ where $n_j$ is the left end point of the interval $J_j(V_m^{\mathcal V})$. Hence the set we are interested in corresponds to the set $D$ in Lemma~\ref{lem:vol} and each set as in \eqref{eqn:smallsets} corresponds to the set $E$ as in Lemma~\ref{lem:vol}. Thus, if necessary enlarging the constant $c_0$ appeared earlier, using Lemma~\ref{lem:vol} we see that once we divide $Y$ into the sets of the form as in \eqref{eqn:smallsets} we only need to keep 
$$\le c_0 e^{\frac{h_{\max}(\T)}{2}\ell}=c_0 e^{h_{\max}(\T)(\ell-\frac{|V_m^{\mathcal V}|}{2})}$$
 many of them. Hence, we conclude that the set $P_{\mathcal  O}(\mathcal V)$ can be covered by
$$\le c_0^{m+n'} \exp\left(h_{\max}(\T)\left[(b-1+\ell)-\frac{(|V_1^{\mathcal V}|+\cdots +|V_m^{\mathcal V}|)}{2}\right]\right)$$
pre-images under $\T^{b+\ell-1}$ of the sets of the form
$$\T^{b+\ell-1}(x_0)u^+(t)B_{\eta/2}^{U^+}a^{-b-\ell+1}B_{\eta/2}^{U^-L}a^{b+\ell-1}.$$
Now, we let $b=N$ to obtain that the set $P_{\mathcal O}(\mathcal V)$ can be covered by 
$$\le c_0^{k+k'} \exp\left(h_{\max}(\T)\left[N-\frac{(|V_1^{\mathcal V}|+\cdots +|V_k^{\mathcal V}|)}{2}\right]\right)$$
many Bowen $N$-balls.
On the other hand, the proof of Lemma~\ref{lem:Q_{M,N}} suggests that $k$ and hence $k'$ are bounded above by
$$\frac{N}{ 2\lfloor\frac{2\log M}{h_r+h_s}\rfloor}+1 .$$
Thus, the set $P_{\mathcal O}(\mathcal V )$ can be covered by 
$$\ll c_0^{\frac{h_{\max}(\T)N}{\log M}}e^{h_{\max}(\T)(N-\frac{|\mathcal V|}{2})}$$
translates of Bowen $N$-balls, which completes the proof.
\qed
\section{Proof of Theorem~\ref{thm:dim}}
\label{sec:dim}
We recall that $a_*=\max\{a_1,a_2,\dots,a_{r+s}\}$ and $D=\dim U^+$. One can easily obtain the following lemma.
\begin{lem}
\label{lem:balls}
For any $N \in \N$, the set $a^N B_{\eta/2}^{U^+}a^{-N} B_{\eta/2}^{U^-L}$ can be covered by $\ll e^{[D a_*-h_{\max}(\T)]N} $ translates of $B_{\frac{\eta}{2} e^{-a_*N}}^{U^+}B_{\frac{\eta}{2}}^{U^-L}$.
\end{lem}
The proof of Proposition~\ref{prop:main} together with Lemma~\ref{lem:balls} at once give
\begin{prop}
\label{prop:dimmain}
The set $P(\mathcal V) $ can be covered by 
$$\ll_M c_0^{\frac{h_{\max}(\T)}{\log M}N}e^{(D N a_* -\frac{h_{\max}(\T) |\mathcal V|}{2})}$$
 translates of $B_{\frac{\eta}{2} e^{-a_*N}}^{U^+}B_{\frac{\eta}{2}}^{U^-L}$ in $X$ for some universal constant $c_0 \ge 1$. 
\end{prop}
\begin{proof}
From the proof of Proposition~\ref{prop:main} we know that the set $P(\mathcal V) $ can be covered by 
$$\ll_M c_0^{\frac{h_{\max}(\T)}{\log M}N}e^{h_{\max}(\T)(N -\frac{ |\mathcal V|}{2})}$$ translates of $a^N B_{\eta/2}^{U^+}a^{-N} B_{\eta/2}^{U^-L}$ in $X$. Thus, Lemma~\ref{lem:balls} finishes the proof.
\end{proof}
For any fixed $N\in \N$ and for any $x\in X$ we associate $\mathcal V_x \subset [0,N-1]$ such that for any $n\in [0,N-1]$
$$\T^n(x) \in X_{\ge M} \text{ if and only if } n \in \mathcal V_x. $$
As in the introduction, let $\nu$ be a measure on $X$ of dimension $d$ in the unstable direction. For any $\delta>0$ we note that $\nu(B_{\frac{\eta}{2} e^{-a_*N}}^{U^+}B_{\frac{\eta}{2}}^{U^-L})\ll e^{-a_* (d-\delta)N}$. Using Proposition~\ref{prop:main} together with Lemma~\ref{lem:P_{M,N}} and Lemma~\ref{lem:Q_{M,N}} it is easy to obtain the following (cf. \cite[Lemma~6.2]{EinKad}).
\begin{lem}
\label{lem:kappa}
For any $\delta>0, \kappa \in[0,1]$ and for any $N,M\ge 1$ large, we have
\begin{multline*}
\nu({x \in X_{< M}: |\mathcal V_x|> \kappa N}) \\
\ll_{M,\delta} \exp\left(\left[Da_*-\frac{h_{\max}(\T) }{2}\kappa -(d-\delta)a_*+O\left(\frac{\log \log M}{\log M}\right)\right]N\right).
\end{multline*}
\end{lem}
\begin{proof}[Proof of Theorem~\ref{thm:dim}]
We follow the proof of \cite[Theorem~1.6]{EinKad}. The conclusion of the theorem is trivial when $d\le D-\frac{h_{\max}(\T)}{2a_*}$ so that we may assume  $d>D-\frac{h_{\max}(\T)}{2a_*}$. We first estimate an upper bound for $\mu_N(X_{\ge M})$ when $M, N\ge 1$ large. We have
\begin{align*}
\mu_N(X_{\geq M})&=\frac{1}{N}\sum_{n=0}^{N-1}\nu(\T^{-n}(X_{\geq M}))\\
&=\frac{1}{N}\sum_{n=0}^{N-1}\nu(X_{< M} \cap \T^{-n}(X_{\geq M}))+\frac{1}{N}\sum_{n=0}^{N-1}\nu(X_{> M} \cap \T^{-n}(X_{\geq M}))\\
&\le\frac{1}{N}\sum_{n=0}^{N-1}\nu(X_{< M} \cap \T^{-n}(X_{\geq M}))+\nu(X_{\ge M}).
\end{align*}
It suffices to estimate is $\frac{1}{N}\sum_{n=0}^{N-1}\nu(X_{< M} \cap \T^{-n}(X_{\geq M}))$. 
For this, we note that
\begin{multline*}
\frac{1}{N}\sum_{n=0}^{N-1}\nu(X_{\le M} \cap \T^{-n}(X_{\geq M}))\\
=\frac{1}{N}\sum_{n=0}^{N-1} \sum_{W \subset [0,N]} \nu(\{x \in X_{< M}: V_x=W\}\cap \T^{-n}(X_{\ge M})),
\end{multline*}
where the term $\nu(\{x \in X_{< M}: V_x=W\}\cap \T^{-n}(X_{\ge M}))$ is either 0 or is equal to $\nu(\{x \in X_{< M}: V_x=W\})$. Switching the order of summation yields 
\begin{align*}
&=\frac{1}{N}\sum_{W \subset [0,N-1]}|W|\nu(\{x \in X_{< M}: V_x=W\})\\
&=\frac{1}{N}\sum_{i=1}^{N}i\nu(\{x \in X_{< M}: |V_x|=i\})\\
&= \frac{1}{N}\sum_{i=1}^{\lfloor \kappa N \rfloor}i\nu(\{x \in X_{< M}: |V_x|=i\})+\frac{1}{N}\sum_{i=\lceil \kappa N \rceil }^{N}i \nu(\{x \in X_{< M}: |V_x|=i\})\\
&\le \frac{1}{N}\lfloor \kappa N \rfloor\nu(X_{< M})+\frac{1}{N}N\nu(\{x \in X_{< M}: |V_x|>\kappa N\})
\end{align*}
Let $K(M,\delta)>0$ be the implicit constant appeared in Lemma~\ref{lem:kappa}. Then we obtain
$$\frac{1}{N}\sum_{n=0}^{N-1}\nu(X_{<M} \cap \T^{-n}(X_{\geq M}))\le \kappa+K(M,\delta) e^{\left(Da_*-\frac{h_{\max}(\T) }{2}\kappa -(d-\delta)a_*+O\left(\frac{\log \log M}{\log M}\right)\right)N}.$$
Therefore we get that
\begin{equation}
\mu_N(X_{\ge M})\le \epsilon(M)+\kappa+K(M,\delta)e^{\left(Da_*-\frac{h_{\max}(\T) }{2}\kappa -(d-\delta)a_*+O\left(\frac{\log \log M}{\log M}\right)\right)N}.
\end{equation}
By assumption we have $d>D-\frac{h_{\max}(\T)}{2a_*}$ so that $2a_*(D-d)/(h_{\max}(\T))<1$. Now, for any $\kappa \in (2a_*(D-d)/(h_{\max}(\T)),1]$ we may pick $\delta>0$ small enough so that
$$Da_*-\frac{h_{\max}(\T) }{2}\kappa -(d-\delta)a_*+O\left(\frac{\log \log M}{\log M}\right)<0$$
for sufficiently large $M$. Thus, for any $\epsilon>0$ we may choose $M$ sufficiently large so that
  $$\mu_N(X_{\ge M})\le \kappa+\epsilon$$
  which gives in the limit that $\mu(X)>1- \kappa.$ This holds for any $\kappa > 2a_*(D-d)/(h_{\max}(\T))$. Thus,
  $$\mu(X)\ge 1- \frac{2a_*(D-d)}{h_{\max}(\T)}.$$
\end{proof}

\bibliographystyle{plain}
\bibliography{mybib}
\end{document}